\documentclass[a4paper, 11pt]{amsart}
\usepackage[
  margin=1.9cm,
  includefoot,
  footskip=30pt,
]{geometry}
\usepackage{amsmath,amssymb,amscd}
\usepackage{amsfonts}
\usepackage[english]{babel}
\usepackage[leqno]{amsmath}
\usepackage{amssymb,amsthm}
\usepackage{mathrsfs} 
\usepackage{amscd}
\usepackage{enumerate}
\usepackage{epsfig}
\usepackage{relsize}
\usepackage{layout}
\usepackage[usenames,dvipsnames]{xcolor}
\usepackage[backref=page]{hyperref}
\usepackage{tikz}
\hypersetup{
 colorlinks,
 citecolor=Green,
 linkcolor=Red,
 urlcolor=Blue}
\usepackage[matrix,arrow,tips,curve]{xy}
\input{xy}
\xyoption{all}
\definecolor{darkgreen}{rgb}{0.0, 0.7, 0.0}
\definecolor{purple}{rgb}{0.5, 0.0, 0.5}
\definecolor{red}{rgb}{0.8, 0.2, 0.0}

\newtheorem{thm}{Theorem}[section]
\newtheorem{bthm}{Theorem}

\newtheorem{lemma}[thm]{Lemma}

\newtheorem{prop}[thm]{Proposition}

\numberwithin{equation}{section}
\theoremstyle{definition}

\theoremstyle{remark}
\newtheorem{remark}[thm]{Remark}

\newcommand{\Z}{\mathbb{Z}}

\newcommand{\Pic}{\operatorname{Pic}}

\DeclareMathOperator{\Hom}{{Hom}}
\DeclareMathOperator{\Ext}{{Ext}}

\def \P{\mathbb{P}}

\def\I{{\mathcal J}}

\def \E{\mathcal E}

\def\O{\mathcal O}

\def\M0{\mathcal M^0}

\newcommand{\Id}{{\operatorname{Id}}}

\DeclareMathOperator{\codim}{{codim}}

\begin{document}

\title[On varieties with Ulrich twisted normal bundles]{On varieties with Ulrich twisted normal bundles}

\author[A.F. Lopez]{Angelo Felice Lopez*}

\dedicatory{\normalsize  \ $\dagger$ Dedicated to the memory of Alberto Collino$^1$} 

\thanks{* Research partially supported by  PRIN ``Advances in Moduli Theory and Birational Classification'' and GNSAGA-INdAM}

\address{\hskip -.43cm Dipartimento di Matematica e Fisica, Universit\`a di Roma
Tre, Largo San Leonardo Murialdo 1, 00146, Roma, Italy. e-mail {\tt lopez@mat.uniroma3.it}}


\thanks{1. A personal note is written in section \ref{note}.}

\begin{abstract} 
We characterize smooth irreducible varieties with Ulrich twisted normal bundle.
\end{abstract}

\maketitle

\section{Introduction}
\label{intro}

In this note we consider the problem of understanding when the normal bundle of a projective variety has a twist that makes it an Ulrich vector bundle.

Recall that if $X \subseteq \P^N$ is a smooth irreducible variety of dimension $n \ge 1$, we say that a vector bundle $\E$ on $X$ is Ulrich if $H^i(\E(-p))=0$ for all $i \ge 0$ and $1 \le p \le n$.

The study of Ulrich vector bundles is closely related with several areas of commutative algebra and algebraic geometry, and often gives interesting consequences on the geometry of $X$ and on the cohomology of sheaves on $X$ (see for example in \cite{es, b1, cmp} and references therein). 

Now, the normal bundle $N_{X/\P^N}$ is associated with the embedding and it is therefore natural to ask when it has a twist that makes it an Ulrich vector bundle. If so, since Ulrich vector bundles are semistable, this would connect with the long-studied problem of varieties with semistable normal bundle. On the other hand, since $N_{X/\P^N}(-1)$ is globally generated, the natural question to be asked is: for which $k \ge 1$ we have that $N_{X/\P^N}(-k)$ is an Ulrich vector bundle?

Certainly, there are some easy examples. If $X$ is a linear space in $\P^N$, then $N_{X/\P^N}(-1) \cong \O_{\P^n}^{\oplus (N-n)}$ is clearly an Ulrich vector bundle on $X$. Also, if $X \subset \P^3$ is a curve such that $H^0(N_{X/\P^3}(-2))=0$, then also $H^1(N_{X/\P^3}(-2))=0$, since $\chi(N_{X/\P^3}(-2))=0$. Thus again $N_{X/\P^3}(-1)$ is an Ulrich vector bundle. This is actually an occurrence of a bit less trivial family of examples, namely $n$-dimensional varieties $X \subset \P^{n+2}, 1 \le n \le 3$  such that $H^j(N_{X/\P^{n+2}}(-2-j))=0$ for $0 \le j \le n-1$, see Lemma \ref{spe}. Interesting examples of this type, but not all of them (see below and Remark \ref{altre}), are linear standard determinantal curves in $\P^3$, surfaces in $\P^4$ and threefolds in $\P^5$, see Proposition \ref{lin} and \cite[Thm. 3.6]{kmr1}.

As it turns out, the above are the only examples, as we will show in the ensuing

\begin{bthm} 
\label{main}

\hskip 3cm

Let $X \subset \P^N$ be a smooth irreducible variety of dimension $n \ge 1$ and let $k$ be an integer. 

Then $N_{X/\P^N}(-k)$ is an Ulrich vector bundle if and only if $k=1$ and $X$ is one of the following:
\begin{itemize}
\item [(i)] $X=\P^n$ embedded linearly in $\P^N$, or 
\item [(ii)] $1 \le n \le 3, N = n+2$ and $X \subset \P^{n+2}$ is a variety such that $N_{X/\P^{n+2}}(-1)$ is $0$-regular, or, equivalently, with $H^j(N_{X/\P^{n+2}}(-j-2))=0$ for $0 \le j \le n-1$. 
\end{itemize}
Moreover, in the latter case, if $X$ does not contain a line, then $N_{X/\P^{n+2}}(-1)$ is very ample.
\end{bthm}

The very ampleness statement is an application of recent results in \cite{ls}. 

We remark that the family of curves $X$ in $\P^3$ satisfying $H^0(N_{X/\P^3}(-2))=0$ is very large (see for example \cite{eh, be}) and there is not much hope to classify them. Moreover $N_{X/\P^3}(-1)$ is semistable, but not always stable by \cite[Thm. 10]{e}.

More generally, on any $X \subset \P^N$, one could investigate the question of whether there exists a line bundle $L$ on $X \setminus {\rm Sing}(X)$ such that $N_{X/\P^N}(L)$ is an Ulrich sheaf. As far as we know this has been considered for the first time by Kleppe and Mir\'o-Roig, who proved it possible for linear standard determinantal schemes \cite[Thm. 3.6]{kmr1}.

Aside for the normal bundle, given $X \subseteq \P^N$ a smooth irreducible variety, one can also consider when the various bundles $\Omega_X(k), T_X(k), {\Omega_{\P^N}}_{|X}(k)$ and ${T_{\P^N}}_{|X}(k)$ are Ulrich vector bundles. We show in Proposition \ref{main2}, that, with the exception of $T_X(k)$, a simple classification can be obtained. The case $T_X(k)$ appears to be more challenging, even though, varieties with Ulrich tangent bundle have been recently classified in \cite{bmpt} and varieties with $T_X(1)$ Ulrich have been classified in \cite{lr}.

Finally, one could consider the same type of problems for other classes of  bundles similar to Ulrich bundles, such as instanton bundles as defined in \cite{c2}, where the author deals with the case $k=0$.

\section{Examples}

As in \cite[Cor.~2.3]{es}, for rank $2$ bundles with $c_1(\E) = K_X + (n+1)H$, less vanishings are needed to verify the property of being Ulrich. 

\begin{lemma}
\label{spe}
Let $X \subseteq \P^N$ is a smooth irreducible variety of dimension $n \ge 1$. Let $\E$ be a rank $2$ vector bundle on $X$ with $c_1(\E) = K_X + (n+1)H$. Then $\E$ is an Ulrich vector bundle if and only if $\E$ is $0$-regular if and only if $H^j(\E(-j-1))=0$ for $0 \le j \le n-1$.
\end{lemma}
\begin{proof}
As is well known, the definition of Ulrich vector bundle given in \cite{es} is equivalent to $H^i(\E(-p))=0$ for all $i \ge 0$ and $1 \le p \le n$ (this follows, for example, by \cite[Prop.~2.1]{es} and \cite[Thm.~2.3]{b1}). 

Now we have that $\E \cong \E^*(K_X + (n+1)H)$, hence the first equivalence is just \cite[Cor.~2.3]{es}. Similarly, by Serre's duality, $\E$ is $0$-regular if and only if $h^{n-i}(\E(-(n-i)-1))=0$ for all $i > 0$, that is $H^j(\E(-j-1))=0$ for $0 \le j \le n-1$.
\end{proof}

We produce examples of varieties as in (ii) of Theorem \ref{main}. They are a linear standard determinantal schemes. It is well known that there are smooth and irreducible ones by \cite[Thm. 6.2]{ps}. 

\begin{prop}
\label{lin}

Let $n \in \{1, 2, 3\}$, let $s \ge 2$ and let $X_s \subset \P^{n+2}$ be a smooth irreducible ACM $n$-dimensional variety with resolution 
\begin{equation}
\label{reso}
0 \to \O_{\P^{n+2}}(-s-1)^{\oplus s} \to \O_{\P^{n+2}}(-s)^{\oplus (s+1)} \to \I_{X_s/\P^{n+2}} \to 0.
\end{equation}
Then $N_{X_s/\P^{n+2}}(-1)$ is an Ulrich vector bundle and it is very ample if $s \ge 3$.
\end{prop}
\begin{proof}
The fact that $N_{X_s/\P^{n+2}}(-1)$ is an Ulrich vector bundle follows by \cite[Thm. 3.6]{kmr1}. We offer a proof for completeness's sake. 

Observe that, for any $0 \le j \le n-1$, we have by \cite[Remark 2.2.6]{kl} that
\begin{equation}
\label{kl}
H^j(N_{X_s/\P^{n+2}}(-j-2)) \cong \Ext^{j+1}(\I_{X_s/\P^{n+2}}, \I_{X_s/\P^{n+2}}(-j-2)).
\end{equation}
Applying $\Hom(-, \I_{X_s/\P^{n+2}}(-j-2))$ to \eqref{reso} we get the exact sequence
$$\Ext^j(\O_{\P^{n+2}}(-s-1)^{\oplus s}, \I_{X_s/\P^{n+2}}(-j-2)) \to \Ext^{j+1}(\I_{X_s/\P^{n+2}}, \I_{X_s/\P^{n+2}}(-j-2)) \to$$ 
$$\to \Ext^{j+1}(\O_{\P^{n+2}}(-s)^{\oplus (s+1)}, \I_{X_s/\P^{n+2}}(-j-2))$$
that is
$$H^j(\I_{X_s/\P^{n+2}}(s-j-1))^{\oplus s} \to \Ext^{j+1}(\I_{X_s/\P^{n+2}}, \I_{X_s/\P^{n+2}}(-j-2)) \to H^{j+1}(\I_{X_s/\P^{n+2}}(s-j-2))^{\oplus (s+1)}.$$
On the other hand, $H^j(\I_{X_s/\P^{n+2}}(s-j-1))=H^{j+1}(\I_{X_s/\P^{n+2}}(s-j-2))=0$ for $0 \le j \le n-1$ by \eqref{reso}.
Hence $\Ext^{j+1}(\I_{X_s/\P^{n+2}}, \I_{X_s/\P^{n+2}}(-j-2))=0$ and therefore $N_{X_s/\P^{n+2}}(-1)$ is an Ulrich vector bundle by \eqref{kl} and Lemma \ref{spe}.

Now assume that $s \ge 3$. To see that $N_{X_s/\P^{n+2}}(-1)$ is very ample it is enough, by \cite[Thm.~1]{ls}, to show that $X_s$ does not contain a line. Let $d = \deg X_s$ and let $g$ be the sectional genus of $X_s$. Since, when $n \ge 2$, a general hyperplane section of $X_s$ has the same resolution, we get by \cite[Rmk.~1]{e} that
\begin{equation}
\label{dp}
d= \frac{s(s+1)}{2}, g = 1 + \frac{s(2s^2+3s+1)}{6}-s(s+1).
\end{equation}
Note that $d \ge 6$, hence we are done if $n=1$. On the other hand, if $n=3$ and $X_s$ contains a line $L$, then picking a general hyperplane $H$ containing $L$, we get a surface $X_s \cap H$ containing $L$. Therefore we will be done if we show that when $n=2$ there is no line on $X_s$. Assume therefore that $n=2$ and that $S:=X_s$ contains a line $L$. By \cite[Thm.~4.1]{kmr2} we know that $\Pic(S)$ is generated by the hyperplane section $H$ and by $K_S$, hence there are two integers $a, b$ such that $L \sim aH+bK_S$. This gives the equations
\begin{equation}
\label{pic}
H \cdot (aH+bK_S)=1, -2 = (aH+bK_S)[aH+(b+1)K_S].
\end{equation}
Setting
$$\alpha = (H \cdot K_S)^2 - d K_S^2 \ \hbox{and} \ \beta = 2d+1+H \cdot K_S$$
it easily follows from \eqref{pic} that
\begin{equation}
\label{qua}
\alpha b^2 + \alpha b - \beta =0.
\end{equation}
Using \eqref{dp} and the standard relation between the invariants of surfaces in $\P^4$ \cite[Appendix A, Ex.~4.1.3]{h2}, one gets that
\begin{equation}
\label{inv}
H \cdot K_S = \frac{s(4 s^2- 9 s-13 )}{6}, K_S^2 = \frac{s (21s^3-118s^2+87s+226)}{24}
\end{equation}
and therefore
\begin{equation}
\label{inv2}
\alpha = \frac{s^2(s^4+3s^3+s^2-3s-2)}{144}, \beta = \frac{4s^3-3s^2-7s+6}{6}.
\end{equation}
Solving \eqref{qua} it is easily seen that $b \le 1$. On the other hand, $b=0$  gives the contradiction $\beta=0$. Therefore $b \le -1$. Now $H-L$ is base-point-free, hence $(H-L)^2 \ge 0$, that is $-2-L \cdot K_S=L^2 \ge 2-d$, and then
$$d(aH+bK_S) \cdot K_S = dL \cdot K_S \le d(d-4).$$ 
Now, using that $ad=1-bH \cdot K_S$ by \eqref{pic}, we deduce that
$$\alpha b \ge -d(d-4) + H \cdot K_S$$
and therefore
$$-\alpha \ge \alpha b \ge -d(d-4) + H \cdot K_S.$$
Replacing the values in \eqref{inv} and \eqref{inv2} we find a contradiction if $s \ge 5$. On the other hand, \eqref{qua} has no integer solutions for $s=3,4$ and we are done.
\end{proof}

\begin{remark}
\label{altre}
As observed in the introduction, for $n=1$ is it clear that there are many examples of curves that satisfy (ii) of Theorem \ref{main} but are not linear standard determinantal. In fact, there are even subcanonical ones, see Remark \ref{sottoc}.
On the other hand, for $n \in \{2, 3\}$, we do not know examples of varieties as in (ii) of Theorem \ref{main}, aside from linear standard determinantal ones. 
We have checked all papers with examples of surfaces $X$ in $\P^4$ and $3$-folds $X$ in $\P^5$ of low degree (such as in \cite{ra1, ra2}, \cite{dp} and several other papers) and, aside from linear standard determinantal ones, none of them has $N_{X/\P^{n+2}}(-1)$ being an Ulrich vector bundle. An explanation is given below.
\end{remark}
Let $n \in \{2, 3\}$ and let $X \subset \P^{n+2}$ be a smooth irreducible $n$-dimensional variety such that $N_{X/\P^{n+2}}(-1)$ is an Ulrich vector bundle (or, equivalently, by Lemma \ref{spe}, $X$ is as in (ii) of Theorem \ref{main}).

If $n=3$ and $S$ is a smooth hyperplane section of $X$, we have that $N_{S/\P^4}(-1) \cong N_{X/\P^5}(-1)_{|S}$ is also an Ulrich vector bundle by \cite[(3.4)]{b1}. Thus, in the study of the degree $d$ and the sectional genus $g$, of varieties $X \subset \P^{n+2}$ as above, we can restrict to the case $n=2$. Now, using the formula for $c_2$ in \cite[(2.2)]{c1}, we get that
\begin{equation}
\label{chi}
\chi(\O_X)= \frac{1}{2}(d^2-4d-5g+5).
\end{equation}
In all papers with examples of surfaces in $\P^4$ and $3$-folds in $\P^5$ of low degree (such as in \cite{ra1, ra2}, \cite{dp} and several other papers), there is only one case satisfying \eqref{chi} and not linear standard determinantal, described as follows. Let $T \subset \P^4$ be the elliptic quintic scroll and let $X$ be a surface linked to $T$ in the complete intersection of a cubic and a quintic hypersurface containing $T$. Then $d=10, g=11$ and $\chi(\O_X)=5$ by \cite[Lemma 9.20]{ra1}. Let $C$ be a general hyperplane section of $X$. To show that $N_{X/\P^4}(-1)$ is not an Ulrich vector bundle, since $N_{C/\P^3}(-1) \cong N_{X/\P^4}(-1)_{|C}$, it is enough, by \cite[(3.4)]{b1}, to prove that $H^0(N_{C/\P^3}(-2)) \ne 0$, for then $N_{C/\P^3}(-1)$ is not an Ulrich vector bundle. To this end, by semicontinuity, it is enough to show that $H^0(N_{C/\P^3}(-2)) \ne 0$ when $X$ is a surface linked to $T$ in the complete intersection of a general cubic $F$ and a quintic hypersurface containing $T$. We claim that such an $F$ has isolated singularities. First observe that $\mathcal I_{T/\P^4}(3)$ is $0$-regular.
In fact, consider the exact sequences 
$$0 \to \mathcal I_{T/\P^4}(l) \to \O_{\P^4}(l) \to \O_T(l) \to 0.$$
For $l=-1$ we get that $H^4(\mathcal I_{T/\P^4}(-1))=0$, for $l=0$ we have that $H^3(\mathcal I_{T/\P^4})=H^2(\O_T)=0$, for $l=1$ we see that $H^2(\mathcal I_{T/\P^4}(1))=H^1(\O_T(1))=0$. Now, if $E$ is the elliptic curve section of $T$, we know by Castelnuovo's theorem that $H^1(\mathcal I_{E/\P^3}(2))=0$, hence the exact sequence
$$0 \to \mathcal I_{T/\P^4}(1) \to \mathcal I_{T/\P^4}(2) \to \mathcal I_{E/\P^3}(2) \to 0$$
shows that $H^1(\mathcal I_{T/\P^4}(2))=0$, since $T$ is linearly normal.

Thus we have shown that $\mathcal I_{T/\P^4}(3)$ is $0$-regular, hence globally generated (see for example \cite[Thm.~1.8.5]{laz}). It follows by \cite[Thm.~1.2]{os} that $F$ has isolated singularities and therefore we have that $C$ linked to $E$ in the complete intersection of a smooth cubic $F'$, hyperplane section of $F$, and a quintic hypersurface containing $E$. This gives an inclusion 
$$0 \to \O_C(3H-E) \to N_{C/\P^3}(-2)$$ 
and we will prove that $H^0(\O_C(3H-E)) \ne 0$. To this end, we have the exact sequence
$$0 \to \O_{F'}(-2) \to \O_{F'}(3H-E) \to \O_C(3H-E) \to 0$$
that implies that $H^0(\O_C(3H-E))=H^0(\O_{F'}(3H-E))$ and the exact sequence
$$0 \to \O_{\P^3} \to \mathcal I_{E/\P^3}(3) \to \O_{F'}(3H-E) \to 0$$
gives that
$$h^0(\O_C(3H-E))=h^0(\O_{F'}(3H-E))=h^0(\mathcal I_{E/\P^3}(3))-1=4.$$
This proves that, when $X$ is linked to $T$ in the complete intersection of a cubic and a quintic hypersurface, $N_{X/\P^4}(-1)$ is not an Ulrich vector bundle.

\section{Proof of Theorem \ref{main}}

Before proving the theorem, we deal separately with the special case of codimension $2$ subcanonical varieties.

\begin{prop}
\label{cod2}

Let $X \subset \P^{n+2}$ be a smooth irreducible variety of dimension $n \ge 1$, degree at least $2$ such that $K_X=eH$ for some $e \in \Z$, $H \in |\O_X(1)|$. Then:
\begin{itemize} 
\item[(i)] Either $H^0(N_{X/\P^{n+2}}(-2)) \ne 0$ or there is an integer $t \ge -e-n-2$ such that $H^1(N_{X/\P^{n+2}}(t)) \ne 0$. 
\item[(ii)] If $n \ge 2$, then $N_{X/\P^{n+2}}(-1)$ is not an Ulrich vector bundle.
\end{itemize} 
\end{prop}
\begin{proof}
First, we see that (i) implies (ii). In fact, if $n \ge 2$ and $N_{X/\P^{n+2}}(-1)$ is an Ulrich vector bundle, then $H^0(N_{X/\P^{n+2}}(-2)) = 0$ and $H^1(N_{X/\P^{n+2}}(t)) = 0$ for every $t \in \Z$ since Ulrich vector bundles are ACM (see for example \cite[(3.1)]{b1}), contradicting (i).

As for (i), assume now that $H^0(N_{X/\P^{n+2}}(-2)) = H^1(N_{X/\P^{n+2}}(t)) = 0$ for every integer $t \ge -e-n-2$.
 
Note that $X$ is not a complete intersection, for otherwise $N_{X/\P^{n+2}} \cong \O_X(a) \oplus \O_X(b)$ for some integers $a \ge b \ge 1$ such that $ab \ge 2$, hence $a \ge 2$. But then $H^0(N_{X/\P^{n+2}}(-2))\ne 0$, a contradiction. 

In particular $X$ is non-degenerate. 

Since $X$ is subcanonical, in order to get a contradiction, it is enough to show that $X$ is projectively normal: In fact,  then the Evans-Griffith's theorem \cite[Thm.~2.4]{eg} would imply that $X$ is a complete intersection (see for example \cite[Rmk.~13]{bc}). 

Let $V=H^0(\O_{\P^{n+2}}(1))$, let $P^1(\O_X(1))$ be the sheaf of principal parts and consider, as in \cite[Proof of Thm.~2.4]{ei}, the following commutative diagram
$$\xymatrix{& 0 \ar[d] & 0 \ar[d] & & \\ & N_{X/\P^{n+2}}^*(1) \ar[r]^{\cong} \ar[d] & N_{X/\P^{n+2}}^*(1) \ar[d] & & \\ 0 \ar[r] & \Omega^1_{\P^{n+2}}(1)_{|X} \ar[r]  \ar[d] & V \otimes \O_X \ar[r]  \ar[d] & \O_X(1) \ar[r]  \ar[d]^{\cong} & 0 \\ 0 \ar[r] & \Omega^1_X(1) \ar[r] \ar[d] & P^1(\O_X(1)) \ar[r] \ar[d] & \O_X(1)  \ar[r] \ar[d] & 0 \\ & 0 & 0 & 0 & }.$$
Pick an integer $l \ge 0$. Tensoring the above diagram by $\O_X(l)$ and observing that 
$$P^1(\O_X(1)) \otimes \O_X(l) \cong P^1(\O_X(l+1))$$
by \cite[(2.2)]{ei}, we get the commutative diagram
$$\xymatrix{& V \otimes H^0(\O_X(l)) \ar[d]^{f_l} \ar[dr]^{h_l} & \\ & H^0(P^1(\O_X(l+1))) \ar[r]^{\ \ \ g_l} \ar[d] & H^0(\O_X(l+1)) & \\ & H^1(N_{X/\P^{n+2}}^*(l+1))}$$
Now $H^1(N_{X/\P^{n+2}}^*(l+1)) \cong H^1(N_{X/\P^{n+2}}(l-e-n-2))=0$, hence $f_l$ is surjective for every $l \ge 0$ and so is $g_l$ by \cite[Prop.~2.3]{ei}. It follows that $h_l$ is surjective for every $l \ge 0$ and the commutative diagram
$$\xymatrix{V \otimes H^0(\O_{\P^{n+2}}(l)) \ar[r] \ar[d]^{\Id_V \otimes r_l} & H^0(\O_{\P^{n+2}}(l+1)) \ar[d]^{r_{l+1}} \\ V \otimes H^0(\O_X(l)) \ar[r]^{h_l} & H^0(\O_X(l+1))}$$
shows by induction that $r_l : H^0(\O_{\P^{n+2}}(l)) \to H^0(\O_X(l))$ is surjective for every $l \ge 0$ and we are done.
\end{proof}

\begin{remark}
\label{sottoc}
Instead, there are many subcanonical curves $X \subset \P^3$ with $H^0(N_{X/\P^3}(-2))=0$, hence with $N_{X/\P^3}(-1)$ an Ulrich vector bundle (see for example \cite{be}).
\end{remark}

We now prove our first result.

\begin{proof}[Proof of Theorem \ref{main}]

If $X$ is embedded linearly in $\P^N$, we saw in the introduction that $N_{X/\P^N}(-1)$ is an Ulrich vector bundle. If $X$ is as in (ii) of Theorem \ref{main}, since $c_1(N_{X/\P^{n+2}}(-1))=K_X + (n+1)H$, we see that $N_{X/\P^N}(-1)$ is an Ulrich vector bundle by Lemma \ref{spe}. Moreover, if, in addition, $X$ does not contain a line, then $N_{X/\P^{n+2}}(-1)$ is very ample by \cite[Thm.~1]{ls}.

Vice versa assume that $X \subset \P^N$ is a smooth irreducible variety of dimension $n \ge 1$, degree $d \ge 1$ and such that $N_{X/\P^N}(-k)$ is an Ulrich vector bundle. In particular $H^0(N_{X/\P^N}(-k-1))=0$ and therefore $k \ge 1$, since, as is well known, $N_{X/\P^N}(-1)$ is globally generated.

If $d=1$ we have that $X=\P^n$ embedded linearly in $\P^N$. 

Assume from now on that $d \ge 2$ and let $H \in |\O_X(1)|$. 

We first show that $k=1$ and $\codim_{\P^N} X=2$.

Note that $N-n \ge 2$. In fact, if $N-n=1$ then $N_{X/\P^N}= \O_X(d)$ and to have $\O_X(d-k)$ Ulrich implies that
$H^0(\O_X(d-k-1))=H^n(\O_X(d-k-n))=0$, that is $k \ge d \ge 2$ and $H^0(\O_X(k-2))=0$, a contradiction.

Let $C$ be a general curve section of $X$ and let $g$ be its genus, so that
$$2g-2=[K_X+(n-1)H] \cdot H^{n-1} = K_X \cdot H^{n-1}+(n-1)d$$
and therefore
\begin{equation}
\label{gen}
K_X \cdot H^{n-1} = 2g-2-(n-1)d.
\end{equation}
Since $c_1(N_{X/\P^N}(-k))=K_X+(N+1-kN+kn)H$ we get, using \eqref{gen} that
$$\deg(N_{X/\P^N}(-k)_{|C}) = [K_X+(N+1-kN+kn)H] \cdot H^{n-1} = 2(d+g-1) -(k-1)(N-n)d.$$
Now \cite[Lemma 2.4(iii)]{ch} implies that
$$\deg(N_{X/\P^N}(-1)_{|C}) = (N-n)(d+g-1)$$
and we deduce that
$$(N-n-2)(d+g-1)+(k-1)(N-n)d=0$$
so that $k=1$ and $N-n=2$.

If $1 \le n \le 3$ we have that $X$ is as in (ii) of Theorem \ref{main} by Lemma \ref{spe}. 

To finish the proof, we show that the case $n \ge 4$ does not occur. In fact, if $n \ge 4$, then Barth-Larsen's type theorems (see for example \cite[Thm.~2.2]{h1}) give that $X$ is subcanonical and now Proposition \ref{cod2}(ii) implies that $N_{X/\P^N}(-1)$ is not Ulrich vector bundle, a contradiction. 
\end{proof}

\section{Other standard bundles associated to $X$}

In this section we give a very simple and complete answer to the question, for a smooth irreducible variety $X \subseteq \P^N$, of when $\Omega_X(k), {\Omega_{\P^N}}_{|X}(k)$ and ${T_{\P^N}}_{|X}(k)$ are Ulrich vector bundles for some integer $k$.
 
\begin{prop} 
\label{main2}

Let $X \subseteq \P^N$ be a smooth irreducible variety of dimension $n \ge 1$ and let $k$ be an integer. Then 
\begin{itemize}
\item [(i)] $\Omega_X(k)$ is an Ulrich vector bundle if and only if $k=2$ and $X$ is a line.
\item [(ii)] ${\Omega_{\P^N}}_{|X}(k)$ is an Ulrich vector bundle if and only if $k=2$ and $X$ is a rational normal curve in $\P^N$.
\item [(iii)] ${T_{\P^N}}_{|X}(k)$ is an Ulrich vector bundle if and only if either $N=2, k=-1$ and $X$ is a conic or $N=1, k=-2$ and $X=\P^1$.
\end{itemize}
\end{prop}
\begin{proof}
We  let $g$ to be the sectional genus of $X$ and $d$ its degree. To see (i), the assertion being obvious when $k=2$ and $X$ is a line, assume vice versa that $\Omega_X(k)$ is an Ulrich vector bundle. If $n \ge 2$, since Ulrich vector bundles are ACM by \cite[(3.1)]{b1}, we get the contradiction $H^1(\Omega_X)=0$. If $n=1$, then the Ulrich condition $H^i(\Omega_X(k-1))=0$ for $i \ge 0$ gives that either $g=0$ and $d=1,k=2$ or $g \ge 1$. In the latter case we have that $H^0(\O_X(1-k))=0$, hence $k \ge 2$, but then $H^0(\Omega_X(k-1)) \ne 0$, a contradiction. This proves (i).

Now we deal with (ii) and (iii). 

If $X$ is a rational normal curve in $\P^N$, then, as is well known, ${\Omega_{\P^N}}_{|X} \cong \O_{\P^1}(-N-1)^{\oplus N}$, and therefore
$$H^i({\Omega_{\P^N}}_{|X}(1)) \cong H^i(\O_{\P^1}(-1))^{\oplus N}=0$$
for every $i \ge 0$, so that ${\Omega_{\P^N}}_{|X}(2)$ is an Ulrich vector bundle. 

If $N=2, k=-1$ and $X$ is a conic or if $N=1, k=-2$ and $X=\P^1$, it is easily checked that ${T_{\P^N}}_{|X}(k)$ is an Ulrich vector bundle. 

Vice versa let $\E$ be either ${\Omega_{\P^N}}_{|X}(k)$ or ${T_{\P^N}}_{|X}(k)$ and assume that $\E$ is an Ulrich vector bundle. As is well known, since ${\Omega_{\P^N}}_{|X}(2)$ and ${T_{\P^N}}_{|X}(-1)$ are globally generated, we get that $k \le 2$ in the first case and $k \le -1$ in the second case. Let $C$ be a general curve section of $X$, let $g$ be its genus and $d$ its degree. Set
$$\varepsilon = \begin{cases} -1 & \text{if } \E={\Omega_{\P^N}}_{|X}(k), \\
1 & \text{if } \E={T_{\P^N}}_{|X}(k) \end{cases}$$ 
so that $c_1(\E)=[\varepsilon(N+1) + kN]H$. Then \cite[Lemma 2.4(iii)]{ch} implies that
$$[\varepsilon(N+1) + kN]d = N(d+g-1)$$
that is
\begin{equation}
\label{eq}
N(g-1)+[N(1-\varepsilon-k)-\varepsilon]d = 0.
\end{equation}
If $d=1$ (hence in particular if $N=1$) we deduce that $N=1, X=\P^1$ and either $\varepsilon=-1, k=2$ or $\varepsilon=1, k=-2$, giving rise to the cases in (ii) and (iii). Suppose then that $d \ge 2$, hence also $N \ge 2$. It follows easily that $N(1-\varepsilon-k)-\varepsilon>0$, hence \eqref{eq} implies that $g=0$ and either $\varepsilon=-1, k=2, d=N$ or $\varepsilon=1, k=-1, d=N=2$. In the latter case we have the case of a conic in (iii). To finish the proof, consider the case $\varepsilon=-1, k=2, d=N$. If $n \ge 2$, then the Ulrich condition for ${\Omega_{\P^N}}_{|X}(2)$ gives in particular that $H^1({\Omega_{\P^N}}_{|X})=0$. But then the Euler sequence
$$0 \to {\Omega_{\P^N}}_{|X} \to \O_X(-1)^{\oplus (N+1)} \to \O_X \to 0$$ 
gives a contradiction. Hence $n=1$ and if $X$ spans a $\P^r$ we have that 
$$0=H^0({\Omega_{\P^N}}_{|X}(1))=H^0({\Omega_{\P^r}}_{|X}(1) \oplus \O_X^{N-r})$$ 
so that $r=N$ and $X$ is a rational normal curve in $\P^N$. 
\end{proof}

\section{A personal note about Alberto Collino}
\label{note}

I first met Alberto while I was a graduate student at Brown University in the mid '80's. He was a visiting professor and I learned then the way, I think, he always has been: a very pleasant person in human relationships
and one always available. He made me feel confident. This was confirmed over the years (even recently), when, after coming back to Italy, I visited Torino. He was a complete mathematician and, among his merits, it is not to be underestimated the contribution that he gave to the Italian algebraic geometry. He was part of the generation of important mathematicians that went to the East Cost to follow a Ph.D. program and then brought back to Italy the way of working learned there.

\end{document}